\documentclass{amsart}
\usepackage{amsmath,amsfonts,amssymb}
\usepackage{graphicx}

\newtheorem{theorem}{Theorem}[section]
\newtheorem{lemma}[theorem]{Lemma}

\newtheorem{cor}[theorem]{Corollary}

\theoremstyle{definition}

\theoremstyle{remark}
\newtheorem{remark}[theorem]{Remark}

\numberwithin{equation}{section}

\begin{document}

\title[On the Constant and Extremal Function for Weighted Hardy Inequality in $L_p$]{On the Constant and Extremal Function for Weighted Hardy Inequality in $L_p$}
\thanks{This study is financed by the Bulgarian National Research Fund through Contract KP-06-N62/4}

\author[Ivan Gadjev]{Ivan Gadjev}
 \address{Department of Mathematics and Informatics,
Sofia University ``St. Kliment Ohridski'',
5 James Bourchier Blvd., 
1164 Sofia, Bulgaria}
\email{gadjev@fmi.uni-sofia.bg}

      

\begin{abstract}
  We study the behaviour of the smallest possible constant $d(a,b, p,\epsilon)$    in Hardy inequality
$$
\int_a^b\left(\frac{1}{x}\int_a^xf(t)dt\right)^px^{\epsilon}\,dx\leq
d(a,b,p,\epsilon)\,\int_a^b [f(x)]^px^{\epsilon}\, dx, \quad 2\le p<\infty.
$$
The exact rate of convergence of $d(a,b,p,\epsilon)$ is established and 
the ``almost extremal'' function is found.
\end{abstract}

\maketitle

\section{Introduction and statement of the results}

Between 1919 and 1925, in the series of papers  \cite {H1919, H1920, H1925} G. H. Hardy established the following inequality
\begin{equation}\label{eq01}
\int_0^\infty\left(\frac{1}{x}\int_0^xf(t)dt\right)^p\,dx\leq
\left(\frac{p}{p-1}\right)^{p}\,\int_0^\infty f^p(x)dx
\end{equation}
where $f$ is such that $f(x)\geq0$ for $x\in (0,\infty)$ and $f^p$ is integrable over $(0,\infty)$. 
The inequality \eqref{eq01} was firstly stated and proved in \cite{H1925}.  

The first weighted version of \eqref{eq01} appeared shortly after the original Hardy inequality. Let $p>1$ and $\epsilon<p-1$. Then for all
measurable functions $f(x)\geq 0$ the next inequality holds
\begin{equation*}\label{I2}
\int_0^\infty\left(\frac{1}{x}\int_0^xf(t)dt\right)^px^\epsilon\,dx\leq
\left(\frac{p}{p-1-\epsilon}\right)^{p}\,\int_0^\infty f^p(x)x^\epsilon\, dx.
\end{equation*}

Later it was extended to what is called the {\em general Hardy integral inequality}: 
\begin{equation}\label{I3}
\left(\int_a^b\left(\int_a^xf(t)dt\right)^q u(x)\,dx\right)^{1/q} \leq
d(a,b,p,q)\left(\int_a^b f^p(x)v(x)\,dx \right)^{1/p}
\end{equation}
where $-\infty\le a<b\le\infty$, $0<q\le\infty$, $1\le p\le\infty$ and $u(x),\,v(x)$ are given weight functions,
i.e. they are measurable and positive almost everywhere in $(a,b)$.

There are many papers investigating different aspects and applications of Hardy's inequalities;  
see for instance 
\cite{KP2003} and the bibliography of \cite{KMP2007}. In most of them
the authors are trying to determine the conditions on the parameters $p$, $q$ and on the weights $u$, $v$ under which the Hardy inequality \eqref{I3} holds for some classes of functions $f$.

In the present paper we investigate a different aspect of the "finite" version of Hardy's inequality  \eqref{I3}
in the special and very important case when $q=p$ and $u(x)=x^{\epsilon-p},\,v(x)=x^\epsilon$, $\epsilon<p-1$, i.e.
 \begin{equation}\label{wLp}
\int_a^b\left(\frac{1}{x}\int_a^xf(t)dt\right)^px^\epsilon\,dx\leq
d(a,b,p,\epsilon)\,\int_a^b f^p(x)x^\epsilon\, dx.
\end{equation}
We study the behaviour of the constant $d(a,b,p,\epsilon)$ in \eqref{wLp} and what the extremal function is.

The best results about $d(a,b,p,\epsilon)$ for $p>1$ could be summarised in the following way
(see, for instance \cite{KMP2007} or \cite{KP2003}).
Let
\[
B=\sup_{a<x<b} \left\{(x-a)^{p-1}\left(x^{1-p}-b^{1-p} \right)\right\}
\]
then for the constant $d(a,b,p,\epsilon)$ the next estimations are true
\[
\frac{1}{p-1-\epsilon}B\le d(a,b,p,\epsilon) \le \left(\frac{p}{p-1-\epsilon} \right)^p B.
\]
It is easy to see that only the right estimation gives asymptotically (when $a\rightarrow 0$ or $b\rightarrow \infty$ or both) the exact constant. But not the rate of convergence.

In \cite{DGM} we studied the inequality \eqref{wLp} for $p=2, \,\epsilon=0$ and established the exact constant $d(a,b,2,0)$ and the corresponding extremal function.

\begin{theorem}\cite{DGM}\label{th01}
Let  $a$ and $b$ be any fixed numbers with $0<a<b<\infty$. Then the inequality
\begin{equation}\label{MR01}
\int_a^b\left(\frac{1}{x}\int_a^xf(t)dt\right)^2 dx\leq
\frac{4}{1+4\alpha^2}\,\int_a^bf^2(x)\, dx,
\end{equation}
where $\alpha$ is the only solution of the equation
$$
\tan\left(\alpha\ln\frac{b}{a} \right)+2\alpha=0\ \ \mathrm{in\ the\ interval}\ \ \left(\frac{\pi}{2\ln\frac{b}{a}},\frac{\pi}{\ln\frac{b}{a}} \right),
$$
holds for every $f \in L^2[a,b]$.
Moreover, equality in \eqref{MR01} is attained for 
\[
f_{a,b}(x)=x^{-1/2}\left(2\alpha\cos\left(\alpha\ln \frac{x}{a}\right)+\sin\left(\alpha\ln  \frac{x}{a}\right) \right).
\]
\end{theorem}

This result was extended recently by F. Gesztesy and M. Pang \cite{GP} to the case of
additional power weights to the its multi-dimensional version on spherical shell domains.


 In this paper we investigate the behaviour of the constant $d(a,b,p,\epsilon)$  for $2< p<\infty$ and $\epsilon< p-1$ and  establish very sharp estimates for $d(a,b,p,\epsilon)$, as well as obtain an ``almost extremal'' function for \eqref{wLp}. Our main result is summarised in the following theorem.  

\begin{theorem}\label{th1}
Let $2\le p<\infty$, $0<a<b<\infty$,  $\epsilon< p-1$,  
$b_0=\exp\left[\frac{\sqrt2 \pi\,p(p-1)}{p-1-\epsilon} \right]$
and $\alpha$ is the only solution of the equation
\[
\tan\left(\alpha\ln \frac{b}{a} \right)+\frac{\alpha p}{p-1-\epsilon}=0\quad\mbox{in the interval}\quad
\left(\frac{\pi}{2\ln \frac{b}{a}},\frac{\pi}{\ln \frac{b}{a}} \right).
\]
 Then there exist positive constants $c_1=c_1(p,\epsilon)$ and  $c_2=c_2(p,\epsilon)$ 
depending only on $p$ and $\epsilon$, 
  such that  for every $0<a<b<\infty$ for which $b/a>b_0$ the next estimates
for the constant $d(a,b,p,\epsilon)$ in \eqref{wLp} hold 
\begin{equation}\label{MR1}
\left(\frac{p}{p-1-\epsilon} \right)^p  \left(1+\frac{c_1}{\ln^2 \frac{b}{a}}\right)^{-1}\!\!
\le d(a,b,p,\epsilon)
\le \left(\frac{p}{p-1-\epsilon} \right)^p  \left(1+\frac{c_2}{\ln^2 \frac{b}{a}}\right)^{-1}\!\!.
\end{equation}
Moreover, the function 
\begin{equation}\label{fun}
f^\ast(x)=x^{-(1+\epsilon)/p}\left(\frac{\alpha p}{p-1-\epsilon} \cos\left(\alpha\ln \frac{x}{a}\right)+\sin\left(\alpha\ln  \frac{x}{a}\right)\right)
\end{equation}
is an "almost extremal" function in the sense that 
\begin{equation*}
\int_a^b\left(\frac{1}{x}\int_a^x f^\ast(t)dt\right)^px^{\epsilon}dx
\geq \left(\frac{p}{p-1-\epsilon} \right)^p  \left(1+\frac{c_1}{\ln^2 \frac{b}{a}}\right)^{-1}\,\int_a^b [f^\ast(x)]^p \,x^{\epsilon}dx.
\end{equation*}
\end{theorem}
\begin{remark}
The right inequality in~\eqref{MR1} is true for any $0<a<b<\infty$.
\end{remark}

\begin{remark}
We prove that the inequalities in the equation~\ref{MR1} hold for the next constants $c_1$ and $c_2$:
\begin{equation*}\label{c1}
  c_1=
    \begin{cases}
          \pi^2\left(\frac{p}{p-1-\epsilon} \right)^2 \left(p-\frac{7}{8}\right), \quad\mbox{for}\quad 2\le p\le3\\
        \pi^2\left(\frac{p}{p-1-\epsilon} \right)^2   \left(2p^2-4p+1\right), \quad\mbox{for}\quad 3<p< \infty
        \end{cases}
 \end{equation*} 
 and
\begin{equation*}\label{c2}
c_2 =  \frac{p-1}{2}\left(\frac{p}{p-1-\epsilon}\arctan{\frac{1}{\sqrt{2(p-2)}}}\right)^2.
 \end{equation*} 
But these constants 
  are by no means the best ones. 
     They could be improved in a lot of ways but that could have made the proofs longer and more complicated.
Our goal was to establish the exact rate of convergence and to keep the proofs as simple as possible. 
\end{remark}
An immediate consequence of Theorem~\ref{th1} is the next corollary. 
\begin{cor}\label{cor1}
 When either of the limits relations $a\rightarrow0$,  $b\rightarrow\infty$, or both hold, i.e. $\ln(b/a)\rightarrow\infty$, then
\[
d(a,b,p,\epsilon) \sim 
\left(\frac{p}{p-1-\epsilon} \right)^p -\frac{C}{\ln^2 \frac{b}{a}}.
\]
More precisely, for every $0<a<b<\infty$ for which $b/a>b_0$ the next inequalities hold 
\begin{equation*}
\left(\frac{p}{p-1-\epsilon} \right)^p -\frac{C_1}{\ln^2 \frac{b}{a}} \le d(a,b,p,\epsilon) \le 
\left(\frac{p}{p-1-\epsilon} \right)^p -\frac{C_2}{\ln^2 \frac{b}{a}}.
\end{equation*}
where
\begin{equation*}
  C_1= 
    \begin{cases}
          \pi^2\left(\frac{p}{p-1-\epsilon} \right)^{p+2} \left(p-\frac{7}{8}\right), \quad\mbox{for}\quad 2\le p\le3\\
        \pi^2\left(\frac{p}{p-1-\epsilon} \right)^{p+2}  \left(2p^2-4p+1\right), \quad\mbox{for}\quad 3<p< \infty
        \end{cases}
 \end{equation*}
 and
\[
C_2 = \frac{2\pi^2(p-1)^2 \left(\frac{p}{p-1-\epsilon}\right)^{p+2}
\Big( \arctan{\frac{1}{\sqrt{2(p-2)}} }\Big)^2}{4\pi^2(p-1)+\Big( \arctan{\frac{1}{\sqrt{2(p-2)}} }\Big)^2}.
\]
\end{cor}
Indeed, the left inequality is obvious. For the right inequality  for 
$b/a>b_0$ we have
\[
 \left(1+\frac{c_2}{\ln^2 \frac{b}{a}}\right)^{-1}
 \le 1-c_2 \left(1+\frac{c_2}{\ln^2 b_0}\right)^{-1}
 \frac{1}{\ln^2 \frac{b}{a}}=1 -\frac{C_2^*}{\ln^2 \frac{b}{a}}
\]
where
\[
C_2^* = c_2 \left(1+\frac{c_2}{\ln^2 b_0}\right)^{-1}=
\frac{2\pi^2(p-1)^2 \left(\frac{p}{p-1-\epsilon}\right)^2
\Big( \arctan{\frac{1}{\sqrt{2(p-2)}} }\Big)^2}{4\pi^2(p-1)+\Big( \arctan{\frac{1}{\sqrt{2(p-2)}} }\Big)^2}.
\]




\begin{remark}
Probably, similar results are still true if $1<p<2$ but a different approach is needed.
\end{remark}

Henceforth, $\alpha$ will always be the only solution of the equation
\[
\tan\left(\alpha\ln \frac{b}{a} \right)+\frac{\alpha p}{p-1-\epsilon}=0
\]
in the interval $\left(\frac{\pi}{2\ln \frac{b}{a}},\frac{\pi}{\ln \frac{b}{a}} \right)$ and $\beta$ will always be $\beta =\frac{p}{p-1-\epsilon}$.

\section{Proof of the right inequality in theorem \ref{th1}}

By simple change of variables and notations it is easy to see that it is enough to prove \eqref{MR1} for the interval $(1,b)$.

Let $q=\frac{p}{p-1}$. From Holder's inequality we have for every two  functions $f(x)\geq0$ and $g(x)> 0, \,x\in(1,b)$ and $f^p(x)$
 and $g^q(x)$ are integrable over (1,b), 
\[
\left(\int_1^x f(t)dt \right)^p   \le\left(\int_1^x g^q(t)dt  \right)^{p-1}\left(\int_1^x \frac{f^p(t)}{g^p(t)}dt  \right).
\]
After multiplying both sides by $x^{\epsilon-p}$, integrating from 1 to $b$ and changing the order of integration in the right side 
we get
\begin{align*}
\int_1^b\left(\frac{1}{x}\int_1^x f(t)dt \right)^p x^\epsilon dx   \le 
\int_1^b\left[\frac{1}{g^p(t)t^\epsilon}\int_t^b\left(\int_1^x g^q(u)du \right)^{p-1}\frac{dx}{x^{p-\epsilon}}  \right]f^p(t)t^\epsilon dt.
\end{align*}
Let us denote for brevity $M(g,t)=g^{-p}(t)t^{-\epsilon}M^*(g,t)$ where
\[
M^*(g,t)=\int_t^b\left(\int_1^x g^q(u)du \right)^{p-1}\frac{dx}{x^{p-\epsilon}}.
\]
Then for every two functions $f(x)\geq 0$ and $g(x)> 0$, $1<x<b$ such that $f^p(x)$ and $g^q(x)$ are integrable the next upper estimation holds
\begin{align*}
\int_1^b\left(\frac{1}{x}\int_1^x f(t)dt \right)^px^\epsilon dx   \le 
\max_{1< t< b}M(g,t)  \int_1^b f^p(t) t^\epsilon dt
\end{align*}
and consequently for every function $g(x)> 0,\,1<x<b$
\[
d(1,b,p,\epsilon)\le \max_{1< t< b}M(g,t). 
\]
Now we want to minimize 
\[
\max_{1< t< b}M(g,t)
\]
 over all functions  $g(x)> 0$ on the interval $(1,b)$ or to find 
\[
\min_{g(x)> 0}\,\max_{1< t< b}\frac{1}{g^p(t)t^\epsilon}\int_t^b\left(\int_1^x g^q(u)du \right)^{p-1}\frac{dx}{x^{p-\epsilon}}.
\]
\begin{remark}
For $g(x)=x^{-(1+\epsilon)/(pq)}$ we obtain the original Hardy inequality. Indeed, we have
\[
\int_1^x g^q(u)du=\frac{p}{p-1-\epsilon}  \left(x^{1-(1+\epsilon)/p}-1\right)<\frac{p}{p-1-\epsilon}x^{1-(1+\epsilon)/p}
\]
\begin{align*}
\int_t^b\left(\int_1^x g^q(u)du \right)^{p/q}\frac{dx}{x^{p-\epsilon}}
<\left(\frac{p}{p-1-\epsilon}\right)^{p/q}
\int_t^b \left(x^{1-(1+\epsilon)/p} \right)^{p/q}\frac{dx}{x^{p-\epsilon}}\\
=\left(\frac{p}{p-1-\epsilon}\right)^{p}\left(t^{\epsilon-(1+\epsilon)/q}-b^{\epsilon-(1+\epsilon)/q} \right)
<\left(\frac{p}{p-1-\epsilon}\right)^{p}t^{\epsilon-(1+\epsilon)/q}\\
=\left(\frac{p}{p-1-\epsilon}\right)^{p} g^p(t)t^{\epsilon}
 \end{align*}
for every $1< t< b$. Consequently $M(g,t)<\left(\frac{p}{p-1-\epsilon}\right)^{p}$ for every $1< t< b$, which means that
\[
\max_{1< t< b}M(g,t)<\left(\frac{p}{p-1-\epsilon}\right)^{p} \quad\mbox{i.e.}\quad d(1,b,p,\epsilon)\le \left(\frac{p}{p-1-\epsilon}\right)^{p} .
\]
\end{remark}

But in order to prove the right inequality in Theorem \ref{th1} we need more complicated choice
of the function $g(x)$.
Let us consider the function $g(x)$, defined in the following way
\begin{equation*}\label{eq200}
g(x)=x^{-(1+\epsilon)/(pq)}\left(\cos(\eta\ln x) \right)^{1/q}
\end{equation*}
where $\eta=\frac{1}{\ln b}\arctan{\frac{1}{\sqrt{2(p-2)}}}$.

 Then we have
\begin{align*}
\int_1^xg^q(u)du&=\frac{\beta}{1+\eta^2\beta^2}\left[x^{1/\beta}(\cos(\eta\ln x)+\eta \beta\sin(\eta\ln x))-1 \right]\\
&<\frac{\beta}{1+\eta^2\beta^2}\left[x^{1/\beta}(\cos(\eta\ln x)+\eta \beta\sin(\eta\ln x))\right]\end{align*}
and for every $1< t< b$
\begin{equation}\label{eqI34}
M^*(g,t)<\beta^{p-1}\int_t^b\left(\frac{\cos(\eta\ln x)+\eta \beta\sin(\eta\ln x)}{1+\eta^2\beta^2} \right)^{p-1}
\frac{dx}{x^{1+1/\beta}}.
\end{equation}

Then by equation \eqref{eqI32} of Lemma \ref{lem3} proved below, there exists a constant $c_2^*= \frac{p-1}{2}\beta^2$, 
 such that for $1\le x\le b$
the next inequality holds:
\begin{equation}\label{eqI35}
\left(\frac{\cos(\eta\ln x)+\eta \beta\sin(\eta\ln x)}{1+\eta^2\beta^2} \right)^{p-1}\le
-\beta\left(1+c_2^*\eta^2 \right)^{-1}x^{1+1/\beta}\left[x^\epsilon g^p(x) \right]'.
\end{equation}

Then from  \eqref{eqI34} and  \eqref{eqI35} it follows that for every $1< t< b$
\begin{align*}
M^*(g,t)<-\frac{\beta^p}{1+c_2^*\eta^2}\int_t^b\left[x^\epsilon g^p(x) \right]'dx
\le \frac{\beta^p}{1+c_2^*\eta^2}t^\epsilon g^p(t)=\beta^p\left(1+\frac{c_2}{\ln^2b}\right)^{-1}t^\epsilon g^p(t)
\end{align*}
where $c_2=c_2^*\left(\arctan{\frac{1}{\sqrt{2(p-2)}}}\right)^2 = \frac{p-1}{2}\left(\beta\arctan{\frac{1}{\sqrt{2(p-2)}}}\right)^2$.
Consequently 
\[
M(g,t)\le  \beta^p\left(1+\frac{c_2}{\ln^2b}\right)^{-1}
=\left(\frac{p}{p-1-\epsilon} \right)^p  \left(1+\frac{c_2}{\ln^2 b}\right)^{-1}.
\]
The last means that 
\[
\max_{1< t< b}M(g,t)\le 
\left(\frac{p}{p-1-\epsilon} \right)^p  \left(1+\frac{c_2}{\ln^2 b}\right)^{-1}
\]
i.e.
\[
 d(1,b,p,\epsilon)\le \left(\frac{p}{p-1-\epsilon} \right)^p  \left(1+\frac{c_2}{\ln^2 b}\right)^{-1}.
\]

\begin{lemma} \label{lem3}
	Let $b>1$, $2< p<\infty$, $q=\frac{p}{p-1}$,  $\epsilon< p-1$, 
	$\eta=\frac{1}{\ln b}\arctan{\frac{1}{\sqrt{2(p-2)}}}$ and
\begin{equation*}\label{eql31}
g(x)=x^{-(1+\epsilon)/pq}\left(\cos(\eta\ln x) \right)^{1/q}.
\end{equation*}
	Then 
	there exists a constant $c=c(p,\epsilon)$, depending only on $p$ and $\epsilon$ such that
	 for $1\le x\le b$  the next inequality holds:
\begin{equation}\label{eqI32}
\left(\frac{\cos(\eta\ln x)+\eta \beta\sin(\eta\ln x)}{1+\eta^2\beta^2} \right)^{p-1}\le
-\beta \left(1+c\eta^2 \right)^{-1}x^{1+1/\beta}\left[x^\epsilon g^p(x) \right]'.
\end{equation}
	\end{lemma}
\begin{remark}
We prove that the inequality~\eqref{eqI32} is true for the constant $c= (p-1)\beta^2/2$.
\end{remark}	
	\begin{proof}
We have
\begin{align*}
\left[x^\epsilon g^p(x) \right]'=-\frac{(\cos(\eta\ln x))^{p-2}}{x^{1+1/\beta}}
\left(\eta (p-1)\sin(\eta\ln x)+\frac{1}{\beta}\cos(\eta\ln x) \right).
\end{align*}
Then the above inequality \eqref{eqI32} is equivalent to
\[
\left(\frac{1+\eta \beta y}{1+\eta^2\beta^2} \right)^{p-1}
=\left(1+\frac{\eta \beta (y-\eta \beta )}{1+\eta^2\beta^2} \right)^{p-1}
\le\frac{1}{1+c\eta^2}\left( 1+\eta\beta (p-1)y\right)
\]
where $y=\tan(\eta\ln x)$. We will apply to the left side the next easily verifiable inequality:
\begin{equation}\label{eql33}
(1+x)^\gamma<\frac{1+x}{1+x-\gamma x}, \quad \gamma>1, \quad -1<x<\frac{1}{\gamma-1}.
\end{equation}
In order to do that we need
\[
-1<\frac{\eta \beta (y-\eta \beta )}{1+\eta^2\beta^2}<\frac{1}{p-2}
\]
which is equivalent to
\[
0<\frac{1+\eta \beta y}{1+\eta^2\beta^2}<\frac{p-1}{p-2}.
\]
Left inequality is obvious and the right inequality is equivalent to
\[
y<\frac{1+(p-1)\eta^2 \beta^2}{(p-2)\eta\beta}.
\]
From $\eta=\frac{1}{\ln b}\arctan{\frac{1}{\sqrt{2(p-2)}}}$ we have 
\[
y=\tan(\eta\ln x)<\frac{1}{\sqrt{2(p-2)}}<\frac{2\sqrt{p-1}}{p-2}
=\frac{2\eta \beta \sqrt{p-1}}{\eta\beta(p-2)}<\frac{1+(p-1)\eta^2 \beta^2}{(p-2)\eta\beta}.
\]
Applying \eqref{eql33} for $x=\frac{\eta \beta (y-\eta \beta )}{1+\eta^2\beta^2}$ and $\gamma=p-1$ we get
\[
\left(\frac{1+\eta \beta y}{1+\eta^2\beta^2} \right)^{p-1}
<\frac{1+\eta \beta y}{1-(p-2)\eta \beta y+(p-1)\eta^2\beta^2}.
\]
Now we need to prove that there is a constant $c>0$ such that
\[
\frac{1+\eta \beta y}{1-(p-2)\eta \beta y+(p-1)\eta^2\beta^2}
\le\frac{1}{1+c\eta^2}\left( 1+\eta\beta (p-1)y\right).
\]
After some simplifications this inequality is equivalent to
\[
(p-1)\beta^2-c +\left((p-1)^2\beta^2-c \right)\eta\beta y
\geq (p-1)(p-2)\beta^2y^2.
\]
But since $y<\frac{1}{\sqrt{2(p-2)}}$ i.e. $(p-1)(p-2)\beta^2y^2\le \frac{p-1}{2}\beta^2$ it is enough  to prove that there is a constant $c>0$ such that
\[
\frac{p-1}{2}\beta^2-c +\left((p-1)^2\beta^2-c \right)\eta\beta y\geq 0.
\]
This inequality is true, for instance, if $c= \frac{p-1}{2}\beta^2$. The lemma is proved.

\end{proof}

\section{Proof of the left inequality in theorem \ref{th1} }
By changing the order of integration, we write the left side of \eqref{wLp} for
 $a=1$ and $f(x)>0,\,\,1<x<b$ in the following way
\begin{align*}
\int_1^b\left(\frac{1}{x}\int_1^xf(t)dt\right)^p \!x^\epsilon\,dx=\int_1^b M(t)f^p(t)t^\epsilon\,dt
\end{align*}
where
\begin{align*}
M(t)=\frac{1}{t^\epsilon f^{p-1}(t)}\int_t^b\left(\int_1^x f(u)du \right)^{p-1}\frac{dx}{x^{p-\epsilon}}. 
\end{align*}
Obviously
\[
d(1,b,p,\epsilon)\geq \min_{1< t< b}M(t).
\]
Then for the function $f^*(x)$ defined in \eqref{fun} for $a=1$ we have
\[
\int_1^x f^*(u)du=\beta x^{1/\beta}\sin(\alpha\ln x)
\]
and
\begin{align*}\label{eqI7}
\int_t^b\left(\int_1^x f^*(u)du \right)^{p-1}\frac{dx}{x^{p-\epsilon}}
=\beta^{p-1}\int_t^b (\sin(\alpha\ln x))^{p-1}\frac{dx}{x^{1+1/\beta}}.
\end{align*}
\begin{remark}
	The function $f^*$ defined by \eqref{fun} is well defined and $f^*(x)>0$. Indeed, if 
	$\alpha\ln x\in \left(0,\frac{\pi}{2}\right]$ it is obvious. If 
	$\alpha\ln x\in \left(\frac{\pi}{2},\pi \right)$ since the function
	$h(x)=\alpha \beta\cos(\alpha\ln x)+\sin(\alpha\ln x)$ 
	is decreasing and $h(b)=0$ it follows that $h(x)>h(b)=0$, i.e. $h(x)>0$ for $1\le x\le b$ and
	 consequently $f^*(x)>0$.
	\end{remark}

Then  from Lemma \ref{lem4} proved below 
 it follows that  there exists a constant $c_1^*$:
  \begin{equation*}
  c_1^*= 
    \begin{cases}
           \beta^2\left(p-\frac{7}{8}\right), \quad\mbox{for}\quad 2\le p\le3\\
         \beta^2 \left(2p^2-4p+1\right), \quad\mbox{for}\quad 3<p< \infty
        \end{cases}
 \end{equation*}
  such that for every $b>b_0$ where  $ b_0=e^{\sqrt2 \pi  \beta (p-1)}$
 \begin{align*}
\int_t^b\left(\int_1^x\! f^*(u)du \right)^{p-1}\!\!\!\frac{dx}{x^{p-\epsilon}}
\geq- \frac{\beta^p } {1+c_1^*\alpha^2}\int_t^b\!\left[x^\epsilon(f^*(x))^{p-1} \right]'dx
=\frac{\beta^p\, t^\epsilon[ f^*(t)]^{p-1}} {1+c_1^*\alpha^2}.
\end{align*}
Consequently 
\begin{align*}
\min_{1< t< b}M(t)\geq \frac{\beta^p} {1+c_1^*\alpha^2}\quad\mbox{i.e.}\quad
d(1,b,p,\epsilon)
\geq  \frac{\beta^p} {1+c_1^*\alpha^2}\geq  \beta^p\left(1+\frac{c_1}{\ln^2 \frac{b}{a}}\right)^{-1}
\end{align*}
where 
\begin{equation*}
  c_1=\pi^2c_1^*= 
    \begin{cases}
          \pi^2  \beta^2\left(p-\frac{7}{8}\right), \quad\mbox{for}\quad 2\le p\le3\\
        \pi^2 \beta^2 \left(2p^2-4p+1\right), \quad\mbox{for}\quad 3<p< \infty.
        \end{cases}
 \end{equation*}
 The left inequality in Theorem \ref{th1} is proved.

\begin{lemma} \label{lem4}
	Let  $p\geq 2$, $\epsilon<p-1$, 
	 $ b_0=e^{\sqrt2 \pi  \beta (p-1)}$ and the function $f^*$ be defined by \eqref{fun}. Then 
 there exists a constant $c=c(p,\epsilon)>0$, depending only on $p$ and $\epsilon$ such that 
  for every $b>b_0$ the next inequality holds 
  \begin{align}\label{eq41}
(\sin(\alpha\ln x))^{p-1}\geq 
- \frac{\beta x^{1+1/\beta}\left[x^\epsilon (f^*(x))^{p-1} \right]'} {1+c\alpha^2}.
\end{align}
\end{lemma}
\begin{remark}
We prove that the inequality~\ref{eq41} is true for the constant
 \begin{equation*}
  c= 
    \begin{cases}
           \beta^2\left(p-\frac{7}{8}\right), \quad\mbox{for}\quad 2\le p\le3\\
         \beta^2 \left(2p^2-4p+1\right), \quad\mbox{for}\quad 3<p< \infty.
        \end{cases}
 \end{equation*}
  \end{remark}	
		
\begin{proof}
Since
 \begin{align*}
&\beta x^{1+1/\beta}\left[x^\epsilon (f^*(x))^{p-1} \right]'\\
&=(\alpha \beta\cos(\alpha\ln x)+\sin(\alpha\ln x))^{p-2} \big(\alpha\beta(p-2)\cos(\alpha\ln x)-\left(1+\alpha^2\beta^2(p-1) \right)\sin(\alpha\ln x) \big)
\end{align*}
 we need to prove that there exists a constant $c$ such that   the next inequality is true
\begin{align}\label{eqI66}
\notag&\frac{1+c\alpha^2}{1+\alpha^2\beta^2( p-1)} (\sin(\alpha\ln x))^{p-1}\\ 
&\geq  (\alpha \beta\cos(\alpha\ln x)+\sin(\alpha\ln x))^{p-2}
\left(\sin(\alpha\ln x) -\frac{\alpha \beta(p-2)}{1+\alpha^2\beta^2( p-1)}\cos(\alpha\ln x)\right).
  \end{align}
  Firstly we will prove that there is a constant $c^*>0$ such that 
   \begin{align}\label{eqI6}
\notag&\left(1+c^*\alpha^2 \right) (\sin(\alpha\ln x))^{p-1}\\ 
&\geq  (\alpha \beta\cos(\alpha\ln x)+\sin(\alpha\ln x))^{p-2}
\left(\sin(\alpha\ln x) -\frac{\alpha \beta(p-2)}{1+\alpha^2\beta^2( p-1)}\cos(\alpha\ln x)\right).
  \end{align}
	We consider two cases.

	\textbf{Case 1.} $\alpha\ln x\in \left[\frac{\pi}{2},\pi \right)$.\\	
	Let $\alpha\ln x=\pi-\phi,\, 0<\phi\le\pi/2$. Then\eqref{eqI6} is equivalent to 
	\begin{align*}
\left(1+c^*\alpha^2 \right)&(\sin\phi)^{p-1}\\
&\geq(-\alpha \beta\cos(\phi)+\sin\phi)^{p-2}
\left(\sin\phi +\frac{\alpha \beta(p-2)}{1+\alpha^2\beta^2( p-1)}\cos\phi\right)\end{align*}
	or
	\begin{align*}
\left(1+c^*\alpha^2 \right)y^{p-1}\geq \left(y-\alpha \beta \right)^{p-2}
\left(y+\frac{\alpha\beta(p-2)}{1+\alpha^2\beta^2(p-1)} \right)
\end{align*}
where $y=\tan\phi$. Since $y>\alpha \beta$ in this case, we have from Bernoulli's inequality
\begin{align*}
\frac{y^{p-1}}{\left(y-\alpha \beta \right)^{p-2}}
=\left(y-\alpha \beta \right)\left(\frac{y}{y-\alpha \beta} \right)^{p-1}
=\left(y-\alpha \beta \right)\left(1+\frac{\alpha\beta}{y-\alpha \beta} \right)^{p-1}\\
\geq \left(y-\alpha \beta \right)\left(1+\frac{\alpha\beta(p-1)}{y-\alpha \beta} \right)
=y+\alpha\beta(p-2)>y+\frac{\alpha\beta(p-2)}{1+\alpha^2\beta^2(p-1)}.
\end{align*}

	\textbf{Case 2.} $\alpha\ln x\in \left[0,\frac{\pi}{2} \right)$.\\
		In this case \eqref{eqI6} is equivalent to
	\begin{align}\label{I61}
\left(1+c^*\alpha^2 \right)y^{p-1}\geq \left(y+\alpha \beta \right)^{p-2}
\left(y-\frac{\alpha\beta(p-2)}{1+\alpha^2\beta^2(p-1)} \right)
\end{align}
	where $y=\tan(\alpha\ln x)$.
	If $	y\le \frac{\alpha\beta (p-2)}{1+\alpha^2\beta^2(p-1)}$ then
	\eqref{I61} is obvious.
	Let
	$y> \frac{\alpha\beta (p-2)}{1+\alpha^2\beta^2(p-1)}$.
Also we have $\alpha\le \pi/\ln b$ and for every $b>b_0$:   
$\ln b> \ln b_0=\sqrt 2 \pi  \beta (p-1)$
and consequently
\begin{equation}\label{eq3.2-2}
\alpha^2\beta^2(p-1)<\frac{1}{2(p-1)}.
\end{equation}
		
	\textbf{Case 2.1.} $2\le p\le 3$\\
	We need to prove that there is a constant $c^*>0$ such that
     \begin{align*}
1+c^*\alpha^2\geq \left(1+\frac{\alpha \beta}{y} \right)^{p-2} 
\left(1-\frac{\alpha\beta(p-2)}{1+\alpha^2\beta^2(p-1)}\frac{1}{y} \right).
\end{align*}	
	Since
	\[
	\left(1+\frac{\alpha \beta}{y} \right)^{p-2} \le 1+\frac{\alpha\beta(p-2)}{y}
	\]
	it is enough to prove that 
	\begin{align*}
1+c^*\alpha^2\geq \left(1+\frac{\alpha\beta(p-2)}{y}\right)
\left(1-\frac{\alpha\beta(p-2)}{1+\alpha^2\beta^2(p-1)}\frac{1}{y} \right)
\end{align*}
	or simplifying it
	\begin{align*}
\frac{c^*y}{p-2}+\frac{\beta^2(p-2)}{1+\alpha^2\beta^2(p-1)}\frac{1}{y}
\geq \frac{\alpha \beta^3(p-1)}{1+\alpha^2\beta^2(p-1)}.
	\end{align*}
	But since 
	\begin{align*}
\frac{c^*y}{p-2}+\frac{\beta^2(p-2)}{1+\alpha^2\beta^2(p-1)}\frac{1}{y}
\geq 2\beta\sqrt{\frac{c^*}{1+\alpha^2\beta^2(p-1)}}     
	\end{align*}
it is enough to prove that there is a constant $c^*$ such that
\begin{align*}
 2\beta\sqrt{\frac{c^*}{1+\alpha^2\beta^2(p-1)}} \geq  
  \frac{\alpha \beta^3(p-1)}{1+\alpha^2\beta^2(p-1)}  
	\end{align*}
i.e.
\begin{align*}
c^* \geq  \frac{\beta^2(p-1)}{4}
  \frac{\alpha^2 \beta^2(p-1)}{1+\alpha^2\beta^2(p-1)}.  
	\end{align*}

From~\eqref{eq3.2-2} we have
\[
\frac{\alpha^2 \beta^2(p-1)}{1+\alpha^2\beta^2(p-1)}=1-\frac{1}{1+\alpha^2\beta^2(p-1)}
 < 1-\frac{1}{1+\frac{1}{2(p-1)}}= \frac{1}{2p-1},
\]
so, it is enough to prove that there is a constant $c^*$ such that
\begin{align*}
c^* \geq \frac{\beta^2(p-1)}{4(2p-1)},
 	\end{align*}	
i.e. by taking	
\begin{align*}
c^* =  \frac{\beta^2(p-1)}{4(2p-1)}
 	\end{align*}	
we complete the proof of the inequality~\eqref{eqI6}  in this case.

\textbf{Case 2.2.} $p> 3$\\
In this case we need to prove that there is a constant $c^*>0$ such that
	\begin{align*}
\left(1+c^*\alpha^2\right)\left(\frac{y}{y+\alpha \beta}\right)^{p-2}
\geq 1-\frac{\alpha\beta(p-2)}{1+\alpha^2\beta^2(p-1)}\frac{1}{y} .
\end{align*}	
	From Bernoulli's inequality
	\begin{align*}
\left(\frac{y}{y+\alpha \beta}\right)^{p-2}
\geq 1-\frac{ \alpha\beta(p-2)}{y+\alpha \beta }. 
\end{align*}
	So, it is enough to prove that 
	\begin{align*}
  \left(1+c^*\alpha^2\right)\left( 1-\frac{\alpha \beta(p-2)}{y+\alpha \beta} \right)
	\geq 1-\frac{\alpha \beta(p-2)}{(1+\alpha^2\beta^2( p-1))y}.
\end{align*}
Simplifying it
\[
c^*y\geq \frac{\beta^2(p-2)(\alpha\beta y(p-1)-1)}{\left(1+\alpha^2\beta^2(p-1)\right)y}
+c^*\alpha\beta(p-3).
\]
Since at the end of the r.h.s
\[
\alpha\beta y(p-1)-1<\alpha\beta y(p-1)\quad \mbox{and}\quad y> \frac{\alpha\beta (p-2)}{1+\alpha^2\beta^2(p-1)}
\]
it is enough to prove that there is a constant $c^*$ such that
\[
\frac{c^*(p-2)}{1+\alpha^2\beta^2(p-1)}\geq \frac{\beta^2(p-1)(p-2)}{1+\alpha^2\beta^2(p-1)}
+c^*(p-3)
\]
or
\[
c^*\left( 1-\alpha^2\beta^2(p-1)(p-3) \right)    \geq \beta^2(p-1)(p-2).
\]
And since $1-\alpha^2\beta^2(p-1)(p-3)>1-\alpha^2\beta^2(p-1)^2>1/2$
 by taking 
\[
c^*=   2\beta^2(p-1)(p-2)
\]
we complete the proof of the inequality~\eqref{eqI6} with a constant 
 \begin{equation*}\label{eqlem32}
  c^*=
  \begin{cases}
         \frac{\beta^2(p-1)}{4(2p-1)}, \quad\mbox{for}\quad 2\le p\le3\\
         2\beta^2(p-1)(p-2), \quad\mbox{for}\quad 3<p< \infty.
        \end{cases}
 \end{equation*}
  Now inequality~\eqref{eqI66} and consequently inequality~\eqref{eq41} are true for any constant 
 $  c\geq\beta^2(p-1)+c^*\left(1+\alpha^2\beta^2(p-1) \right)$. From~\eqref{eq3.2-2} we have 
 \[
 1+\alpha^2\beta^2(p-1)<\frac{2p-1}{2p-2},
 \] 
so it is enough to take $ c=\beta^2(p-1)+c^*(2p-1)/(2p-2)$, i.e. 
 	\begin{equation*}
  c= 
    \begin{cases}
          \beta^2\left(p-\frac{7}{8}\right), \quad\mbox{for}\quad 2\le p\le3\\
         \beta^2 \left(2p^2-4p+1\right), \quad\mbox{for}\quad 3<p< \infty.
        \end{cases}
 \end{equation*}
 The lemma is proved.

\end{proof}

\bibliographystyle{amsplain}

\end{document}